\documentclass[11pt, reqno]{amsart}
\pdfoutput=1
\usepackage[a4paper,scale=0.75,twoside=false]{geometry} 
\usepackage{amssymb}
\usepackage{amsthm}
\usepackage{amsmath}
\usepackage{tikz}
\usepackage{graphicx}
\usepackage{xcolor}
\usepackage{csquotes}
\usepackage{graphicx}
\usepackage{csquotes}
\usepackage{float}
\usepackage{soul}
\usepackage{enumitem}
\usepackage{amsaddr}
\usepackage{doi}

\definecolor{blue}{rgb}{0.0, 0.313, 0.608}
\usepackage{hyperref}
\hypersetup{
    colorlinks,
    breaklinks=true
    linkcolor={blue},
    citecolor={blue},
    urlcolor={orange}
}
\usepackage{cleveref}

\usepackage[style=ext-numeric-comp,
date=year,
giveninits=true,
articlein=false,
backend=biber,
sorting=nyt]{biblatex}

\AtEveryBibitem{\clearfield{issn}}
\AtEveryBibitem{\clearfield{isbn}}
\DeclareFieldFormat[article,inbook,incollection,inproceedings,patent,thesis,unpublished,misc]{title}{{#1\isdot}}
\DeclareFieldFormat[misc,inbook]{date}{(#1)\isdot}

\usepackage{csquotes}

\newtheorem{theorem}{Theorem}[section]
\newtheorem{lemma}[theorem]{Lemma}
\newtheorem{proposition}[theorem]{Proposition}
\newtheorem{corollary}[theorem]{Corollary}

\theoremstyle{definition}

\newtheorem{remark}[theorem]{Remark}

\newlist{hypothesis}{enumerate}{10}
\setlist[hypothesis]{label*=\emph{(\roman*)}}

\crefname{hypothesisi}{hypothesis}{hypotheses}
\Crefname{hypothesisi}{Hypothesis}{Hypotheses}

\numberwithin{equation}{section}

\newcommand{\norm}[1]{\left\lVert#1\right\rVert}
\newcommand{\R}{\mathbb{R}}

\newcommand{\hrefemail}[1]{\href{mailto:#1}{#1}}

\usepackage[skip=0pt plus1pt, indent=22pt]{parskip}

\title{Periodically modulated solitary waves of the CH--KP-I equation}

\thanks{$^\ddagger$Corresponding author}

\author{Dag Nilsson$^\dagger$}
\email{$^\dagger$\hrefemail{dag.nilsson@math.lu.se}}
\address{Centre for Mathematical Sciences, Lund University, Box 118, 221 00 Lund}

\author{Douglas Svensson Seth$^\ddagger$}
\email{$^\ddagger$\hrefemail{douglas.s.seth@ntnu.no}}
\address{Department of Mathematical Sciences, Norwegian University of Science and Technology, 7491 Trondheim, Norway}

\author{Yuexun Wang$^*$}
\email{$^*$\hrefemail{yuexunwang@lzu.edu.cn}}
\address{School of Mathematics and Statistics, Lanzhou University, 730000 Lanzhou City, P. R. China}
\begin{document}
\begin{abstract}
		We consider the CH--KP-I equation. For this equation we prove the existence of steady solutions, which are solitary in one horizontal direction and periodic in the other. We show that such waves bifurcate from the line solitary wave solutions, i.e. solitary wave solutions to the Camassa--Holm equation, in a dimension-breaking bifurcation. This is achieved through reformulating the problem as a dynamical system for a perturbation of the line solitary wave solutions, where the periodic direction takes the role of time, then applying the Lyapunov-Iooss theorem.
\end{abstract}

\maketitle
	
\renewcommand{\theequation}{\arabic{section}.\arabic{equation}}
\setcounter{equation}{0}
\section{Introduction}

The Kadomtsev--Petviashvili equation (KP equation)
\begin{equation}\label{eq:KP}
(u_t+uu_x+u_{xxx})_x+\sigma u_{yy}=0,
\end{equation}
where $\sigma=\pm1$, is a two-dimensional extension of the KdV equation
\begin{equation*}
u_t+uu_x+u_{xxx}=0.
\end{equation*}
The KP equation was first derived in \cite{KP1970} as a model for mainly unidirectional waves in a plasma. Since then \cref{eq:KP} has been found to model waves in other physical settings as well, in particular water waves \cite{AblowitzSegur1979}. The KP equation is called KP-I when $\sigma=-1$ and KP-II when $\sigma=1$. In the context of water waves, the KP-I appears as a model for strong surface tension, while $\sigma=1$ corresponds to the case of weak surface tension. Similarly to how the KP equation can be seen as a two-dimensional extension of the KdV equation, the Camassa--Holm--KP equation (CH--KP equation)
\begin{equation}\label{eq:CH--KP}
\left((1-\partial_x^2)u_t+3uu_x+2\kappa u_x-2u_xu_{xx}-uu_{xxx}\right)_x+\sigma u_{yy}=0,
\end{equation}
where $\kappa>0$, $\sigma=\pm1$, is a two-dimensional extension of the Camassa-Holm equation 
\begin{equation}\label{eq:CH}
(1-\partial_x^2)u_t+3uu_x+2\kappa u_x-2u_xu_{xx}-uu_{xxx}=0.
\end{equation}
Just as in the case of the KP equation, for $\sigma=-1$ \cref{eq:CH--KP} is called the CH--KP-I equation and for $\sigma=1$ it is called CH--KP-II and these correspond to weak and strong surface tension respectively. \Cref{eq:CH--KP} was first introduced in \cite{Chen2006} in the setting of nonlinear elasticity. More recently, the CH--KP-II equation was derived from the Euler equations in \cite{Gui_et_al2021}, in which they also proved local well-posedness in Sobolev spaces.

In this paper we are interested in travelling wave solutions of the CH--KP \cref{eq:CH--KP}, that is solutions of the form $u(t,x,y)=\phi(x-ct,y)$. A travelling wave solution which decays to zero as $|x-ct|+y\rightarrow \infty$ is called a solitary wave solution. In particular, for the Camassa--Holm equation $u(t,x)=Q(x-ct)$ is a solitary wave if $Q(x-ct)\rightarrow 0$ as $|x-ct|\rightarrow \infty$. We have the following result regarding existence of solitary waves for the Camassa--Holm equation.
\begin{theorem}[\cite{CamassaHolmHyman1994}]
There exists solitary waves $Q$ of \cref{eq:CH} only for $c>2\kappa$ and each such $c$ determines $Q$ uniquely up to translations.
\end{theorem}

A solitary wave solution $Q$ of \cref{eq:CH} can be regarded as an $y$-independent solution of \cref{eq:CH--KP}. We refer to such solutions as line solitary wave solutions. In recent years there has been several studies into the stability of line solitary waves of the CH--KP equation. Indeed, in \cite{ChenJin2021} the nonlinear transverse instability of line solitary waves for the CH--KP-I equation was established and in \cite{Geyer_et_al2023} they proved linear stability of line solitary waves for the CH--KP-II equation.
Moreover, the spectral stability of one-dimensional small-amplitude periodic travelling waves with respect to two-dimensional perturbations was investigated in \cite{Chen_et_al_2024}, for a family of $b$-KP equations which includes the CH--KP equation.

\subsection*{Main result}
In the present paper we are interested in proving existence of periodically modulated solitary waves for the CH--KP-I equation. These are solutions with a solitary wave profile in the $x$-direction and periodic in the $y$-direction. We will prove that such waves bifurcate from the line solitary wave solutions in a dimension-breaking bifurcation, i.e, a bifurcation for which a spatially inhomogeneous solution emerges from a solution which is homogeneous in at least one spatial variable \cite{HaragusKirchgassner1996}. 

The CH--KP-I equation is given by
\begin{equation}\label{eq:CH--KP-I}
	\left((1-\partial_x^2)u_t+3uu_x+2\kappa u_x-2u_x u_{xx}-u u_{xxx}\right)_x-u_{yy}=0.
\end{equation}
We make the travelling wave ansatz $u(x,y,t)=v(x-ct,y)$. \Cref{eq:CH--KP-I} then becomes
\begin{equation*}
	\left(-cv_x+cv_{xxx}+3vv_x+2\kappa v_x-2v_xv_{xx}-vv_{xxx}\right)_x-v_{yy}=0,
\end{equation*}
and by using that $-2v_xv_{xx}-vv_{xxx}=-\tfrac{\mathrm{d}}{\mathrm{d}x}(\tfrac{v_x^2}{2}+vv_{xx})$, this equation can be written as 
\begin{equation}\label{eq:travelling CH}
	\left(-cv+cv_{xx}+2\kappa v+\frac{3}{2}v^2-\frac{1}{2}v_x^2-vv_{xx}\right)_{xx}-v_{yy}=0.
\end{equation}
For a fixed wavespeed $c>2\kappa$, let $Q$ be the unique solitary wave solution of the CH equation, so $Q$ satisfies the equation
\begin{equation*}
	-cQ+cQ''+2\kappa Q+\frac{3}{2}Q^2-\frac{1}{2}Q'^2-QQ''=0.
\end{equation*}
In \cref{eq:travelling CH} we make the ansatz $v=\phi+Q$, which gives us the following equation for $\phi$:
\begin{equation*}
	\left[\left(\partial_x(Q-c)\partial_x+Q''-3Q+c-2\kappa\right)\phi\right]_{xx}+\left(\frac{1}{2}\phi_x^2+\phi\phi_{xx}-\frac{3}{2}\phi^2\right)_{xx}+\phi_{yy}=0.
\end{equation*}
Let $\phi=\psi_x$ and integrate with respect to $x$. The resulting equation we get for $\psi$ can be written as 
\begin{equation}\label{psi_eq}
	-L\psi+N(\psi)+\psi_{yy}=0,
\end{equation}
where
\begin{align*}
	L&=-\partial_x M\partial_x,\\
	M&=\partial_x(Q-c)\partial_x+Q''-3Q+c-2\kappa,\\
	N(\psi)&=\left(\frac{1}{2}\psi_{xx}^2+\psi_x\psi_{xxx}-\frac{3}{2}\psi_x^2\right)_x.
\end{align*}
Our main result is
\begin{theorem}
	Let $c>2\kappa$ and let $Q$ be the unique solitary wave solution of \cref{eq:CH} with wave speed $c$. Then there exists an open neighborhood $I\subset \mathbb{R}$ of the origin and a continuously differentiable branch
	\begin{equation*}
		\{(\phi(s),\omega(s))\vert s\in I\}\subset C_b(\mathbb{R};H^4(\mathbb{R}))\times (0,\infty)
	\end{equation*} 
	of solutions of \cref{psi_eq}, which are odd in $x$ and $2\pi/\omega(s)$ periodic in $y$, with $\omega(0)=\sqrt{|\lambda|}$, where $\lambda<0$ is the single negative eigenvalue of the linear operator $L$.
\end{theorem}

In order to prove this theorem we will reformulate the CH--KP-I equation as a dynamical system
\begin{equation}\label{eq:dyn_sys}
	W_y=\mathcal{L}W+\mathcal{N}(W),
\end{equation}
where the periodic variable $y$ plays the role of time (see \cref{eq:lyap}). In the case $\mathcal{L}$ has nonresonant eigenvalues $\pm \mathrm{i}\omega_0$, it follows from the Lyapunov-centre theorem that there exists periodic solutions of \cref{eq:dyn_sys} with frequency close to $\omega_0$. In many situations however the operator $\mathcal{L}$ will have spectrum at $0$, which violates the nonresonance condition of the Lyapunov-Centre theorem. In these instances one can instead use the Lyapunov-Iooss theorem, which relaxes the condition that $\mathcal{L}$ is invertible to only requiring that $\mathcal{L}$ is invertible on the range of $\mathcal{N}$ \cite{Iooss1999}.  

We finish this section by presenting some previous results on periodically modulated solitary waves. Periodically modulated solitary waves have been found for the water wave problem, in \cite{Groves_et_al2002} they proved existence of such waves emerging from a KdV-type solitary wave, for weak surface tension and in \cite{Groves_et_al2016} they proved existence of waves emerging from an NLS-type solitary wave, for weak surface tension. These types of waves have also been found numerically in for infinite depth \cite{MilewskiWang2014} and also for the Davey-Stewartson equation \cite{Groves_et_al_DS2016}. Periodically modulated solitary waves have also been found for various other model equations. NLS-type periodically modulated waves have were found for the Davey--Stewartson equation. In \cite{TajiriMurakami1990} they found explicit periodically modulated solitary waves for the KP-I equation and a family of solutions were found in \cite{HaragusPego1999} for a generalized KP-I equation. More recently, periodically modulated solitary waves were found for the fractional KP equation in \cite{Borluk_et_al2022}.

\section{Spatial dynamics formulation}
We rewrite \cref{psi_eq} as a dynamical system with $y$ playing the role of time, by introducing the new variable $W=(W_1,W_2)=(\psi_y,\psi)$. The resulting dynamical system becomes
\begin{equation}\label{eq:dyn_system}
	W_y=\mathcal{L}W+\mathcal{N}(W),
\end{equation}
where 
\begin{equation*}
	\mathcal{L}=\left(\begin{array}{cc}
		0 & L\\
		1 & 0
		\end{array}\right),\ \mathcal{N}(W)=\left(\begin{array}{c}
		-N(W_2)\\
		0
	\end{array}\right).
\end{equation*}
Denote by \(H^k(\mathbb{R})\) \((k\in \mathbb{Z})\) the standard Sobolev spaces. We will use $||\cdot||_{L^2},\ ||\cdot||_{H^k}$ and $\langle \cdot,\cdot\rangle$ to denote the usual $L^2(\mathbb{R})$, $H^k(\mathbb{R})$ norms and $L^2(\mathbb{R})$ inner product respectively. If $U$ is a subspace of $L^2(\mathbb{R})$ we will use $U^\perp$ to denote the orthogonal complement of $U$ with respect to $\langle\cdot,\cdot\rangle$.
Let
\begin{align*}
	X&=L_{\text{odd}}^2(\mathbb{R})\times H_{\text{odd}}^2(\mathbb{R}),\\
	Y&=H_{\text{odd}}^2(\mathbb{R})\times H_{\text{odd}}^4(\mathbb{R}),
\end{align*}
where the subscript odd denotes
restriction to odd functions in the respective function spaces,
with norms
\begin{align*}
	||W||_X&=||W_1||_{L^2}+||W_2||_{H^2},\\
	||W||_Y&=||W_1||_{H^2}+||W_2||_{H^4}.
\end{align*}
Then $\mathcal{L}$ is a linear operator on $X$ with $D(\mathcal{L})=Y$ and $\mathcal{N}$ is a nonlinear operator from $Y$ to $X$. 

\section{Existence of periodic solutions}
We want to prove existence of periodic solutions of \cref{eq:dyn_system}. In order to do this we will use the Lyapunov-Iooss theorem \cite{Groves_Bagri2015}.
\begin{theorem}\label{thm:lyap}
	Consider the differential equation
	\begin{equation}\label{eq:lyap}
		\dot{W}=\mathcal{L}W+\mathcal{N}(W),
	\end{equation}
	in which $W$ belongs to a Banach space $X$ and suppose that
	\begin{hypothesis}
		\item $\mathcal{L}\colon D(\mathcal{L})\subset X\rightarrow X$ is a densely defined, closed linear operator. \label{hyp:1}
		\item There is an open neighborhood $U$ of the origin in $D(\mathcal{L})$ (regarded as a Banach space equipped with the graph norm) such that $\mathcal{N}\in C_{b,u}^3(U,X)$ and $\mathcal{N}(0)=0,\mathrm{d}{N}[0]=0$.   \label{hyp:2}
		\item \Cref{eq:lyap} is reversible: there exists a bounded operator $S$ on $X$ such that $S\mathcal{L}W=-\mathcal{L}SW$ and $S\mathcal{N}(W)=-\mathcal{N}(SW)$ for all $W\in U$.  \label{hyp:3}
		\item For each $W^*\in U$, the equation 
		\begin{equation}\label{eq:solvability}
			\mathcal{L}W=-\mathcal{N}(W^*),
		\end{equation}
		has a unique solution $W\in D(\mathcal{L})$ and the mapping $W^*\mapsto W$ belongs to $C_{b,u}^3(U,D(\mathcal{L}))$. \label{hyp:4}
	\end{hypothesis}
Suppose further that there exists $\omega_0\in\mathbb{R}$ such that
	\begin{hypothesis}[resume]
		\item $\pm\mathrm{i}\omega_0$ are nonzero simple eigenvalues of $\mathcal{L}$.  \label{hyp:5}
		\item $\mathrm{i}n\omega_0\in\rho(\mathcal{L})$ for $n\in \mathbb{Z}\backslash \{-1,0,1\}$. \label{hyp:6}
		\item $\norm{(\mathcal{L}-\mathrm{i}n\omega_0I)^{-1}}_{X\rightarrow X}=o(1)$ and $\norm{(\mathcal{L}-\mathrm{i}n\omega_0I)^{-1}}_{X\rightarrow D(\mathcal{L})}=O(1)$ as $|n|\rightarrow \infty$. \label{hyp:7}
	\end{hypothesis}
	Under these hypotheses there exists an open neighborhood $I$ of the origin in $\mathbb{R}$ and a continuously differentiable branch $\{(W(s),\omega(s))\}_{s\in I}$ in $C_b^1(\mathbb{R},X)\cap C_b(\mathbb{R},D(\mathcal{L}))$ of solutions to \cref{eq:lyap}, which are $\frac{2\pi}{\omega(s)}$-periodic, with $\omega(0)=\omega_0$.
\end{theorem}
 The last three hypothesis require the most effort to verify and we will therefore initially turn our attention to them. The first four, on the other hand, can be verified in a very brief manner and we do so at the end of this section to complete the proof of the main theorem. In order to verify \cref{hyp:5,hyp:6,hyp:7} we gather some results on the spectra of $M$ and $L$.
\begin{lemma}
	The operator $M\colon H^2(\mathbb{R})\subset L^2(\mathbb{R})\rightarrow L^2(\mathbb{R})$ is self-adjoint and has a simple negative eigenvalue $\lambda_0$, with eigenfunction $\psi_0\in H^2(\mathbb{R})$, a simple eigenvalue $0$ with eigenfunction $Q'$ and the rest of the spectrum is positive and bounded away from $0$.
\end{lemma}
\begin{proof}
	Since $c-Q>2\kappa>0$, it follows from \cite[Chapter 5 Section 6]{Kato} that $M$ is self-adjoint. In order to determine the spectrum of $M$, we follow the argument in \cite{ConstantinStrauss2002}. Let 
	\begin{equation*}
		z=\int_0^x\frac{1}{\sqrt{c-Q}}\ \mathrm{d}\theta,\ \Gamma(z)=(c-Q)^\frac{1}{4}\psi(x).
	\end{equation*}
Then
\begin{align*}
	\psi_x&=\frac{\Gamma_z}{(c-Q)^\frac{3}{4}}+\frac{1}{4}\frac{\Gamma Q'}{(c-Q)^\frac{5}{4}},\\
	\psi_{xx}&=\frac{\Gamma_{zz}}{(c-Q)^\frac{5}{4}}+\frac{\Gamma_z Q'}{(c-Q)^\frac{7}{4}}+\frac{1}{4}\frac{\Gamma Q''}{(c-Q)^\frac{5}{4}}+\frac{5}{16}\frac{\Gamma (Q')^2}{(c-Q)^\frac{9}{4}},
\end{align*}
so
\begin{equation}\label{eq:connection_K_M}
	\begin{aligned}
	M\psi&=(Q-c)\psi_{xx}+Q'\psi_x+(Q''-3Q+c-2\kappa)\psi\\
	&=\frac{1}{(c-Q)^\frac{1}{4}}K\Gamma,
	\end{aligned}
\end{equation}
where 
\begin{equation*}
	K=-\partial_{zz}+q(z)+c-2\kappa,
\end{equation*}
and
\begin{equation*}
	q(z)=\frac{3}{4}Q''-3Q-\frac{1}{16}\frac{(Q')^2}{c-Q}.
\end{equation*}
Since $q$ decays exponentially it follows that $\sigma_{\text{ess}}(K)=[c-2\kappa,\infty)$ and there are at most finitely many eigenvalues located to the left of $c-2\kappa$. The nth eigenvalue, in increasing order, has up to a constant multiple, a unique eigenfunction with precisely $n-1$ zeroes. Due to \cref{eq:connection_K_M} the operator $M$ has the same properties. We note that $MQ'=0$, so $0$ is an eigenvalue with eigenfunction $Q'$ and $Q'$ has exactly one zero, which means that $0$ is the second eigenvalue of $M$. Therefore there is exactly one negative simple eigenvalue $\lambda_0$ of $M$ with eigenfunction $\psi_0\in H^2(\mathbb{R})$. Since $\lambda_0$ is simple and $M$ preserves parities of functions, $\psi_0$ is either even or odd. Due to $\psi_0$ not having any zeroes we conclude that $\psi_0$ must be even.

We next show $\psi_0\in H^2$. Since $\psi_0\in H^1$ is the eigenfunction corresponding to the eigenvalue $\lambda_0$ of the operator $M$,
then one has
\begin{equation}\label{eq:regularity}
\begin{aligned}
\int [(c-Q)\partial_x\psi_0\partial_x\phi+(Q''-3Q+c-2\kappa)\psi_0\phi]\,\mathrm{d} x=\int\lambda_0\psi_0\phi\,\mathrm{d} x
\end{aligned}
\end{equation}
for any $\phi\in H^1$. Set $D^hw(x)=\frac{w(x+h)-w(x)}{h}, \ h\neq 0$.
Taking $\phi=-D^{-h}D^h\psi_0$ and inserting it in \cref{eq:regularity} yields
\begin{equation}\label{eq:regularity-2}
\begin{aligned}
\int [D^h((c-Q)\partial_x\psi_0)D^h\partial_x\phi_0+D^h((Q''-3Q+c-2\kappa-\lambda_0)\psi_0)D^h\phi_0]\,\mathrm{d} x=0.
\end{aligned}
\end{equation}
It follows from \cref{eq:regularity-2} by standard calculations that
\begin{equation*}
\begin{aligned}
	\|D^h\partial_x\psi_0\|_{L^2}\lesssim \|\psi_0\|_{H^1},\quad \forall h\neq 0,
\end{aligned}
\end{equation*}
which implies $\psi_0\in H^2$. 
\end{proof}

We next turn to the operator $L$.
\begin{proposition}\label{prop:L_selfadjoint}
	The operator $L\colon H^4(\mathbb{R})\subset L^2(\mathbb{R})\rightarrow L^2(\mathbb{R})$ is self-adjoint.  
\end{proposition}
\begin{proof}
	Clearly $L$ is symmetric since $M$ is. Let $\tilde{\psi}\in \text{dom}(L^*)$ and let $\psi\in C_0^\infty(\mathbb{R})$. Then
	\begin{align*}
		\langle L^*\tilde{\psi},\psi\rangle&=\langle \tilde{\psi},L\psi\rangle
		=-\langle \tilde{\psi},\partial_x M\partial_x\psi\rangle
		=-\langle \partial_x M\partial_x\tilde{\psi},\psi\rangle,
	\end{align*}
	where we regard $\partial_x M\partial_x\tilde{\psi}$ as an element in $H^{-4}(\mathbb{R})$. Next note that
	\begin{align*}
		|\langle \partial_xM\partial_x\tilde{\psi},\psi\rangle|=|\langle L^*\tilde{\psi},\psi\rangle|
		\leq ||L^*\tilde{\psi}||_{L^2}||\psi||_{L^2},
	\end{align*}
	which by density implies that $\partial_xM\partial_x\tilde{\psi}$ defines a bounded linear functional on $L^2(\mathbb{R})$. By Riesz representation theorem this implies that $\partial_x M\partial_x\tilde{\psi}\in L^2(\mathbb{R})$, which in turn implies by ellipticity that $\tilde{\psi}\in H^4(\mathbb{R})$. Therefore $\text{dom}(L^*)=H^4(\mathbb{R})=\text{dom}(L)$, hence $L$ is self-adjoint. 
\end{proof}
\begin{proposition}\label{prop:spec_L}
	The operator $L$ has a single negative eigenvalue $\lambda$ and the rest of the spectrum is included in $[0,\infty)$.
\end{proposition}
\begin{proof}
	The operator $L$ can be written as 
	\begin{equation*}
		L=\partial_x^2((c-Q)\partial_x^2)-(c-2\kappa)\partial_x^2-\partial_x((Q''-3Q)\partial_x),
	\end{equation*}
	and since $Q$ and its derivatives decay exponentially to $0$, the operator $\partial_x((Q''-3Q)\partial_x)$ is compact. Moreover, $\partial_x^2((c-Q)\partial_x^2)-(c-2\kappa)\partial_x^2$ is a positive operator, which implies that its spectrum is included in $[0,\infty)$. Therefore
	\begin{equation*}
		\sigma_{\text{ess}}(L)=\sigma_{\text{ess}}(\partial_x^2((c-Q)\partial_x^2)-(c-2\kappa)\partial_x^2)\subseteq[0,\infty).
	\end{equation*}
	We proceed to show that $L$ has a single negative eigenvalue.
	Let $\{\psi_n\}_{n\in \mathbb{N}}\subset H^4(\mathbb{R})$ be a sequence such that $\partial_x\psi_n\rightarrow \psi_0\in H^2(\mathbb{R})$. Then 
	\begin{equation}\label{eq:neg_seq}
		\begin{aligned}
		\langle L\psi_n,\psi_n\rangle &=-\langle\partial_x M\partial_x\psi_n,\psi_n\rangle 
		=\langle M\partial_x\psi_n,\partial_x\psi_n\rangle \\
		&\rightarrow_{n\rightarrow \infty} \langle M\psi_0,\psi_0\rangle
		=\lambda_0||\psi_0||^2<0.
		\end{aligned}
	\end{equation}
	The calculation in \cref{eq:neg_seq} implies that $\inf\sigma(L)<0$, so there exists at least one negative eigenvalue. Since
		\begin{equation}
		\langle M\psi,\psi\rangle \geq 0, \text{ for all }\psi\in \{\psi_0\}^\perp,
	\end{equation}
	we have that
	\begin{align*}
		\langle L\psi,\psi\rangle &=\langle M\partial_x\psi,\partial_x\psi\rangle \geq 0,\text{ for all }\psi\in H^4(\mathbb{R}) \text{ such that }\psi_x\in \{\psi_0\}^\perp.
	\end{align*}
	This means that $L$ is nonnegative on a co-dimension one subspace, and so coupled with the previous argument this shows that $L$ has precisely one negative eigenvalue.
\end{proof}
Restricting the operator to odd functions we obtain the following result.
\begin{proposition}\label{prop:odd_1}
	The operator $L\colon H_{\text{odd}}^4(\mathbb{R})\rightarrow L_{\text{odd}}^2(\mathbb{R})$ is self-adjoint and invertible, with bounded inverse
	\begin{equation*}
		L^{-1}\colon L_{\text{odd}}^2(\mathbb{R})\rightarrow H_{\text{odd}}^4(\mathbb{R}).
	\end{equation*}
\end{proposition}
\begin{proof}
	That $L$ is self-adjoint when restricting to odd functions is immediate from \Cref{prop:L_selfadjoint}.
	
	Let $\psi\in\ker{L}$, then $\partial_x M\partial_x\psi=0$ which implies that $M\partial_x\psi=d$ for some constant $d$. However, $M\partial_x\psi\in H_{\text{even}}^1 (\R)$ which implies that $d=0$. Therefore $\partial_x\psi\in \ker{M}$, and $\ker{M}=\text{span}\{Q'\}$ so $\partial_x\psi$ must be a multiple of $Q'$. Since $Q$ is even this means that $\psi$ is even as well, which implies that $\psi=0$, so $L$ is injective. Using that $L$ is injective and self-adjoint we find that
	\begin{equation*}
		\{0\}=\ker{L}=\ker{L^*}=\text{ran}(L)^\perp,
	\end{equation*}
	which implies that $L$ is surjective. Since $L$ is bijective it follows from the open mapping theorem that $L$ is invertible, with bounded inverse
	\begin{equation*}
		L^{-1}\colon L_{\text{odd}}^2(\mathbb{R})\rightarrow H_{\text{odd}}^4(\mathbb{R}).
	\end{equation*}
\end{proof}
By combining \Cref{prop:spec_L,prop:odd_1} we immediately get the following result.
\begin{corollary}\label{prop:spec_L_odd}
	The operator $L\colon H_{\text{odd}}^4(\mathbb{R})\rightarrow L_{\text{odd}}^2(\mathbb{R})$ has precisely one eigenvalue, which is negative and the rest of the spectrum is included in $(0,\infty)$.
\end{corollary}
We now use the above results to investigate the spectrum of $\mathcal{L}$.
\begin{proposition}\label{prop:spec_curlyL}
	Let $\lambda$ be the single negative eigenvalue of $L$. Then $\mathcal{L}\colon Y\subset X\rightarrow X$ has two simple eigenvalues $\pm\mathrm{i}\sqrt{|\lambda|}$ and the rest of the spectrum is included in $\mathbb{R}$. 
\end{proposition}
\begin{proof}
     The resolvent equation $(\mathcal{L}-\mu I)W=W^*$ is equivalent to 
     \begin{equation*}
     	\begin{cases}
     		&(L-\mu^2)W_2=W_1^*+\mu W_2^*,\\
     		&W_1=\mu W_2+W_2^*
     	\end{cases}
     \end{equation*}
     and therefore $\mu\in\rho(\mathcal{L})$ if and only if $\mu^2\in\rho(L)$. We know from \Cref{prop:spec_L_odd} that $\rho(L)=\mathbb{C}\backslash(\{\lambda\}\cup I)$, where $I\subset (0,\infty)$. So $\mu^2\in \rho(L)$ if and only if $\mu\in\mathbb{C}\backslash(\{\pm\mathrm{i}\sqrt{|\lambda|}\}\cup \tilde{I})$, where $\tilde{I}\subset \mathbb{R}$. If we put $W^*=(0,0)$ then we immediately get that $\mu$ is an eigenvalue of $\mathcal{L}$ if and only if $\mu^2$ is an eigenvalue of $L$, which implies that $\mu=\pm\sqrt{|\lambda|}$ and these eigenvalues are simple.
\end{proof}
It follows from \Cref{prop:spec_curlyL} that \cref{hyp:5,hyp:6} of \Cref{thm:lyap} are satisfied. To show that \cref{hyp:7}  is also satisfied, we prove the following result. 

\begin{proposition}
	We have that
	\begin{align}
		\norm{(\mathcal{L}-\mathrm{i}n\sqrt{|\lambda|})^{-1}}_{X\rightarrow X}&\lesssim \frac{1}{\sqrt{|n|}},\label{eq:resolvent_1}\\
		 \norm{(\mathcal{L}-\mathrm{i}n\sqrt{|\lambda|})^{-1}}_{X\rightarrow Y} &\lesssim 1,\label{eq:resolvent_2}
	\end{align}
	for $|n|> 1$.
\end{proposition}
\begin{proof}
	In the proof we will make use of the following resolvent estimates for $L$:
	\begin{align}
	\norm{(L-n^2\lambda)^{-1}}_{L^2\rightarrow L^2}&\lesssim \frac{1}{n^2},\label{eq:aux_est_1}	\\
	\norm{(L-n^2\lambda)^{-1}}_{H^k\rightarrow H^{k+4}}&\lesssim 1,\text{ for all }k\geq 0,\label{eq:aux_est_2}
	\end{align}	
	where \cref{eq:aux_est_1} is due to $L$ being self-adjoint and \cref{eq:aux_est_2} follows from standard elliptic estimates, see e.g. \cite[Chapter 7, Proposition 5.5]{Taylor2023}. By interpolating between \cref{eq:aux_est_1} and \cref{eq:aux_est_2} (for $k=0$ and $k=4$ respectively) we obtain
	\begin{align}
		\norm{(L-n^2\lambda)^{-1}}_{L^2\rightarrow H^2}&\lesssim \frac{1}{|n|}\label{eq:aux_est_3}\\
		\norm{(L-n^2\lambda)^{-1}}_{H^2\rightarrow H^4}&\lesssim \frac{1}{|n|}\label{eq:aux_est_4}.
	\end{align}
	Furthemore we will use that
	\begin{equation}\label{eq:aux_est_5}
		\norm{L}_{H^{k+4}\rightarrow H^k}\lesssim 1 \text{ for all }k\geq 0,
	\end{equation}
	and for $s'>s''>s'''\geq 0$ and $\varepsilon>0$ there exists a constant $D\geq 0$ such that
	\begin{equation*}
		\norm{\psi}_{H^{s''}}\leq \varepsilon \norm{\psi}_{H^{s'}}+D\varepsilon^{-\frac{s''-s'''}{s'-s'''}}\norm{\psi}_{H^{s'''}}
	\end{equation*}	
	(see \cite[Theorem 1.4.3.3]{Grisvard2011}). In particular, for $s'=4,\ s''=2,\ s'''=0$,
	\begin{equation}\label{eq:aux_est_6}
		\norm{\psi}_{H^2}\lesssim \varepsilon \norm{\psi}_{H^4}+\varepsilon^{-1}\norm{\psi}_{L^2}.
	\end{equation}
	
	The inverse of $\mathcal{L}-\mathrm{i}n\sqrt{|\lambda|}$ is given by 
	\begin{equation*}
		(\mathcal{L}-\mathrm{i} n\sqrt{|\lambda|})^{-1}=
		\left(\begin{array}{cc}
			\mathrm{i} n\sqrt{|\lambda|}(L-n^2\lambda)^{-1} & L(L-n^2\lambda)^{-1}\\
			(L-n^2\lambda)^{-1} & \mathrm{i}n \sqrt{|\lambda|}(L-n^2\lambda)^{-1}
			\end{array}\right).
	\end{equation*}
	In order to prove \cref{eq:resolvent_1} we consider
	\begin{equation}\label{eq:norm_1}
	\begin{aligned}
		\norm{(\mathcal{L}-\mathrm{i}n\sqrt{|\lambda|})^{-1}W}_X&\leq \norm{\mathrm{i}n\sqrt{|\lambda|}(L-n^2\lambda)^{-1}W_1}_{L^2}+\norm{L(L-n^2\lambda)^{-1}W_2}_{L^2}\\
&\quad+\norm{(L-n^2\lambda)^{-1}W_1}_{H^2}
		 +\norm{\mathrm{i}n\sqrt{|\lambda|}(L-n^2\lambda)^{-1}W_2}_{H^2},
	\end{aligned}
	\end{equation}
	where
	\begin{align*}
		\norm{\mathrm{i}\sqrt{|\lambda|}(L-n^2\lambda)^{-1}W_1}_{L^2}&\underbrace{ \lesssim}_{\eqref{eq:aux_est_1}}\frac{1}{|n|}\norm{W_1}_{L^2},\\
		\norm{L(L-n^2\lambda)^{-1}W_2}_{L^2}&\underbrace{\lesssim}_{\eqref{eq:aux_est_5}}\norm{(L-n^2\lambda)^{-1}W_2}_{H^4}
			\underbrace{\lesssim}_{\eqref{eq:aux_est_4}}\frac{1}{|n|}\norm{W_2}_{H^2},\\
			\norm{(L-n^2\lambda)^{-1}W_1}_{H^2}&\underbrace{\lesssim}_{\eqref{eq:aux_est_3}}\frac{1}{|n|}\norm{W_1}_{L^2},\\
			\norm{\mathrm{i}\sqrt{|\lambda|}(L-n^2\lambda)^{-1}W_2}_{H^2}&\underbrace{\lesssim}_{\eqref{eq:aux_est_6}}|n|\varepsilon\norm{(L-n^2\lambda)^{-1}W_2}_{H^4}+|n|\varepsilon^{-1}\norm{(L-n^2\lambda)^{-1}W_2}_{L^2}\\
			&\lesssim\frac{1}{\sqrt{|n|}}\norm{W_2}_{H^2},
	\end{align*}
	if we chose $\varepsilon=\frac{1}{\sqrt{|n|}}$ in the last inequality. When using these inequalities in \cref{eq:norm_1} we obtain 
	\begin{equation*}
		\norm{(\mathcal{L}-\mathrm{i}n\sqrt{|\lambda|})^{-1}W}_X\lesssim \frac{1}{\sqrt{|n|}}\norm{W}_X,
	\end{equation*}
	for $|n|>1$, which implies \cref{eq:resolvent_1}.
	
	It remains to prove \cref{eq:resolvent_2}. For this inequality we consider
	\begin{equation}\label{eq:norm_2}
	\begin{aligned}
			\norm{(\mathcal{L}-\mathrm{i}n\sqrt{|\lambda|})^{-1}W}_Y&\leq \norm{\mathrm{i}n\sqrt{|\lambda|}(L-n^2\lambda)^{-1}W_1}_{H^2}+\norm{L(L-n^2\lambda)^{-1}W_2}_{H^2}\\
&\quad+\norm{(L-n^2\lambda)^{-1}W_1}_{H^4}
		 +\norm{\mathrm{i}n\sqrt{|\lambda|}(L-n^2\lambda)^{-1}W_2}_{H^4},
	\end{aligned}
	\end{equation}
	where
	\begin{align*}
		\norm{\mathrm{i}n\sqrt{|\lambda|}(L-n^2\lambda)^{-1}W_1}_{H^2}&\underbrace{\lesssim}_{\eqref{eq:aux_est_3}}\norm{W_1}_{L^2},\\
		\norm{L(L-n^2\lambda)^{-1}W_2}_{H^2}&\underbrace{\lesssim}_{\eqref{eq:aux_est_5}}\norm{(L-n^2\lambda)^{-1}W_2}_{H^6}\underbrace{\lesssim}_{\eqref{eq:aux_est_2}}\norm{W_2}_{H^2},\\
		\norm{(L-n^2\lambda)^{-1}W_1}_{H^4}&\underbrace{\lesssim}_{\eqref{eq:aux_est_2}}\norm{W_1}_{L^2},\\
		\norm{\mathrm{i}n\sqrt{|\lambda|}(L-n^2\lambda)^{-1}W_2}_{H^4}&\underbrace{\lesssim}_\eqref{eq:aux_est_4}\norm{W_2}_{H^2}.
	\end{align*}
	When using these inequalities in \cref{eq:norm_2} we obtain
	\begin{equation*}
		\norm{(\mathcal{L}-\mathrm{i}n\sqrt{|\lambda|})^{-1}W}_Y\lesssim \norm{W}_X,
    \end{equation*}	
    which implies \cref{eq:resolvent_2}.
	\end{proof}

This shows that \cref{hyp:7} of \Cref{thm:lyap} is satisfied and therefore we have verified the last three hypotheses of the Theorem. Here we finish the proof of the main theorem by verifying the remaining hypotheses of \Cref{thm:lyap}, that is, the first four. The operator $\mathcal{L}$ is densely defined and since the resolvent set is nonempty according \Cref{prop:spec_curlyL}, it is also closed, so \cref{hyp:1} is satisfied. \Cref{hyp:2} is also satisfied since $N$ is as smooth operator from $H_{\text{odd}}^4(\mathbb{R})\rightarrow L_{\text{odd}}^2(\mathbb{R})$ and clearly $\mathcal{N}(0)=\mathrm{d}\mathcal{N}[0]=0$. It is straightforward to verify that $S\mathcal{L}W=-\mathcal{L}SW$, $S\mathcal{N}(W)=-\mathcal{N}(SW)$, for the bounded operator $SW=(-W_1,W_2)$, hence \cref{hyp:3} is satisfied. Next consider the solvability condition in \cref{hyp:4}. \Cref{eq:solvability} is equivalent with 
\begin{equation*}
	\begin{cases}
		&LW_2=N(W_2^*),\\
		&W_1=0.
	\end{cases}
\end{equation*}
Since $N(W_2^*)=\left(\frac{1}{2}(W_{2xx}^*)^2+W_{2x}^*W_{2xxx}^*-\frac{3}{2}(W_{2x}^*)^2\right)_x\in L_\text{odd}^2(\mathbb{R})$ for all $W_2^*\in H_{\text{odd}}^4(\mathbb{R})$, it follows from \Cref{prop:odd_1} that 
\begin{equation}\label{eq:solvability_formula}
	W_2=L^{-1}(N(W_2^*))\in H_{\text{odd}}^4(\mathbb{R}),
\end{equation}
The right hand side of \cref{eq:solvability_formula} defines a smooth mapping from $H_{\text{odd}}^4(\mathbb{R})$ to itself, so \cref{hyp:4} is satisfied.
\begin{remark}
	We mention here that we do not use the Lyapunov-Iooss theorem in its full generality since $0\in\rho(\mathcal{L})$ in our case, hence \cref{hyp:4} is automatically satisfied in our case as seen in the above calculations.
\end{remark}
\section*{acknowledgements}
This research was carried out while D. Nilsson was supported by the Knut and Alice Wallenberg foundation and while D. Svensson Seth was supported by the Research Council of Norway project no. 325114. Y. Wang acknowledges the support of grant no. 830018 from China.

\printbibliography

\end{document}